\documentclass[a4paper, 11pt]{amsart}   

\usepackage{mathptmx, amssymb,amscd,latexsym, eulervm}   
\usepackage{amsmath}
\allowdisplaybreaks
\usepackage{amsthm}
\usepackage{mathdots}
\usepackage{setspace}
\usepackage{hyperref}
\hypersetup{
 colorlinks,
 linkcolor={blue!90!black},
 citecolor={red!80!black},
 urlcolor={blue!50!black},
breaklinks=true
}
\usepackage{tabularx}
\usepackage{comment}
\usepackage{amsfonts}
\usepackage{paralist}
\usepackage{aliascnt}
\usepackage[initials, lite]{amsrefs}
\usepackage{amscd}
\usepackage{blkarray}
\usepackage{booktabs}
\usepackage{enumitem}
\usepackage{setspace}
\usepackage[inner=2.5cm,outer=2.5cm, bottom=3.2cm]{geometry}
\usepackage{hyperref}
\usepackage[capitalize,nameinlink]{cleveref}
\usepackage{tikz, tikz-cd}
\usepackage{calligra,mathrsfs}
\usetikzlibrary{matrix}
\usetikzlibrary{arrows,calc}
\usepackage{url}
\usepackage{amssymb,latexsym,amsmath,amsfonts,amsthm,amscd,graphicx,url,color,hyperref,stmaryrd}

\theoremstyle{plain}

\newtheorem{thm}{Theorem}[section]
\newtheorem{prop}[thm]{Proposition}
\newtheorem{cor}[thm]{Corollary}
\newtheorem{lemma}[thm]{Lemma}

\newtheorem{conj}[thm]{Conjecture}

\theoremstyle{definition}
\newtheorem{definition}[thm]{Definition}
\newtheorem{remark}[thm]{Remark}

\newtheorem{example}[thm]{Example}
\newtheorem{problem}[thm]{Problem}
\newtheorem{question}[thm]{Question}

\newcommand{\G}{\Gamma}
\newcommand{\tG}{{\tilde{\Gamma}}}
\newcommand{\mm}{\mathfrak{m}}
\newcommand{\NN}{\mathbb{N}}
\newcommand{\w}{\mathrm{wd}}
\newcommand{\gr}{\mathrm{gr}}
\newcommand{\Ap}{\mathrm{Ap}}
\newcommand{\ord}{\mathrm{ord}}

\newcommand{\ovI}{\overline{I}}
\newcommand{\ovR}{\overline{R}}
\newcommand{\ovP}{\overline{P}}

\newcommand{\wS}{\widehat{S}}
\newcommand{\wI}{\widehat{I}}
\newcommand{\wL}{\widehat{L}}

\newcommand{\Lex}{\mathrm{Lex}}

\newcommand{\HF}{\mathrm{HF}}
\newcommand{\HS}{\mathrm{HS}}

\makeatletter
\@namedef{subjclassname@2020}{
  \textup{2020} Mathematics Subject Classification}
\makeatother

\begin{document}

\author[G.\,Caviglia, A.\,Moscariello, A.\,Sammartano]{Giulio~Caviglia, Alessio~Moscariello and Alessio~Sammartano}
\address{(Giulio Caviglia) Department of Mathematics\\Purdue University\\West Lafayette, IN 47907-2067\\USA}
\email{gcavigli@purdue.edu}
\address{(Alessio Moscariello) Dipartimento di Matematica e Informatica\\Università degli Studi di Catania\\Catania\\Italy}
\email{alessio.moscariello@unict.it}
\address{(Alessio Sammartano) Dipartimento di Matematica \\ Politecnico di Milano \\ Milan \\ Italy}
\email{alessio.sammartano@polimi.it}

\title{Bounds for syzygies of monomial curves}


\begin{abstract}
Let $\G \subseteq \NN$ be a numerical semigroup.
In this paper, we prove an upper bound for the Betti numbers of the semigroup ring of $ \G$ which depends only on the width of $\G$, that is, the difference between the largest and the smallest generator of $\G$.
In this way, we make progress towards a conjecture of Herzog and Stamate.
Moreover,
for 4-generated numerical semigroups, the first significant open case,
we prove  the Herzog-Stamate bound for all but finitely many values of the width.
\end{abstract}

\maketitle

\section{Introduction}

Let $\G  \subseteq \NN$
be a numerical semigroup
with minimal generators $g_0 < g_1 < \cdots < g_\nu$.
Its semigroup ring  is  $R_\G = \Bbbk \llbracket t^\gamma \, : \, \gamma \in \G\rrbracket \subseteq \Bbbk \llbracket t\rrbracket$, where $\Bbbk$ is a field,
that is,
 the complete local ring of the affine monomial curve 
$x_0  = t^{g_0}, x_1=t^{g_1}, \ldots, x_\nu = t^{g_\nu}$
at its singularity at the origin.
Let $P = \Bbbk\llbracket x_0, x_1, \ldots, x_\nu\rrbracket$ be a power series ring, 
then $R_\G$ is isomorphic to the quotient $P/I_\G$, where $I_\G$ is the toric ideal of $\G$, 
that is, the kernel of the surjection defined by $x_i \mapsto t^{g_i}$.

In this paper, we study the Betti numbers of  $R_\G$:
$$
b_i(R_\G) = \dim_\Bbbk \mathrm{Tor}_i^P(R_\G,\Bbbk)
\qquad
\text{for }\, i = 0, 1, \ldots, \nu+1,
$$
that is, the number of $i$-th syzygies in the minimal free resolution of $R_\G$ over $P$.
In particular,
the first Betti number $b_1(R_\G)$ is equal to the number $\mu(I_\G)$  
of minimal generators of  $I_\G$,
i.e., the number of equations of the monomial curve parametrized by $\G$,
while the last Betti number $b_\nu(R_\G)$ is equal to the type of $\G$.
See \cite{MoscarielloSammartano} for a list of open problems on this topic.

It is known that each Betti number is bounded by a function of  the smallest 
generator $g_0$,
since $R_\G$ is a Cohen-Macaulay local ring of multiplicity $g_0$, 
cf. Theorem \ref{ThmValla}.
We are interested in upper bounds in terms of a different invariant of $\G$,
 the  \emph{width}, 
 defined as  $\w(\G) = g_\nu-g_0$.
The motivation for considering this invariant,  introduced in \cite{HerzogStamate},
originates from a result of Vu \cite{Vu} stating that the Betti numbers of the 
``shifted''  semigroups
$ \langle g_0 + j, g_1+j, \ldots, g_\nu+j \rangle $ 
are eventually periodic in  $j$, 
with period equal to $\w(\Gamma)$.
An interesting  consequence of Vu's theorem is  that 
each  $b_i(R_\G)$ takes finitely many values as $\G$ ranges over all semigroups with a fixed width $w$.
This fact is somewhat surprising, as one may expect the Betti numbers to grow arbitrarily large as the multiplicity grows.
Thus, the natural question is: 

\begin{question}\label{QueBoundWidth}
Let $\G$ be a numerical semigroup  and  $R_\G$  its semigroup ring. 
Assuming that the width of $\G$ is  $\w(\G) = w$,
what are the largest possible values of the Betti numbers $b_i(R_\G)$?
\end{question}

This question is wide open;
Vu's proof does not produce any explicit bound, and, in fact, 
no  upper bounds for the Betti numbers are known in terms of the width.
An explicit bound for the first Betti number was proposed by Herzog and Stamate \cite{HerzogStamate}.

\begin{conj}[Herzog-Stamate]\label{ConjectureHS}
Let $\G$ be a numerical semigroup. Then, $$\mu(I_\G) \leq {\w(\G) +1 \choose 2}.$$

\end{conj}

Conjecture \ref{ConjectureHS} is in fact implied by a stronger
 \cite[Conjecture 2.1]{HerzogStamate}, which predicts the same inequality  for the defining ideal of the tangent cone of $R_\G$.
No progress has been made on this conjecture since it was stated.
We  formulate here a natural extension to all the Betti numbers.

\begin{conj}\label{ConjectureBetti}
Let $\G$ be a numerical semigroup. Then, 
$b_i(R_\G) \leq i{\w(\G) +1 \choose i+1}$ for all $i \geq 1$.
\end{conj}

Conjecture \ref{ConjectureHS} is the case $i = 1$ of Conjecture \ref{ConjectureBetti}.
If true, these bounds are sharp, since they are attained, for example, 
by  semigroups of the form $\langle m, m+1, \ldots, 2m-1\rangle$.

In this paper, we prove the following upper bounds.

\begin{thm}\label{TheoremGeneralBounds}
Let $\G$ be a numerical semigroup. 
Then, for all $i \geq 1$ we have
$$
b_i(R_\G) \leq {\w(\G) \choose i} (3e)^{\sqrt{2\w(\G)}}.
$$
where $e$ denotes the base of natural logarithms.
\end{thm}

The bounds of Theorem \ref{TheoremGeneralBounds} are quite far from the expected ones, hence, they are likely far from being sharp.
However, they are the first known bounds  to depend only on the width of the semigroup,
thus, Theorem \ref{TheoremGeneralBounds} is a first step towards answering Question \ref{QueBoundWidth}.

In order to prove Theorem \ref{TheoremGeneralBounds}, we investigate the Hilbert-Samuel function of an artinian reduction of $R_\G$.
By comparing the Ap\'ery set of $\G$ to that of its interval completion $\tG$, 
in Section  \ref{SectionMonomialIdeals} we show that the knowledge of the width of the semigroup imposes several constraints on the Hilbert-Samuel function.
Then, in Section \ref{SectionSyzygies} we obtain the desired bounds by 
estimating the size of the hyperplane section of a lexsegment ideal satisfying these constraints.

A classical result of Herzog \cite{Herzog} states that if $\G$ has at most three generators,
then $b_1(R_\G) \leq 3$ and, as a consequence, $b_2(R_\G) \leq 2$.
By contrast, Bresinsky \cite{Bresinsky} proved that if $\G$ has four or more generators, then the Betti numbers can be arbitrarily large.
Thus, the first interesting instance of the problem of bounding the Betti numbers is the case of 4-generated semigroups.
In Section \ref{SectionFourGenerated}, we refine the analysis  to settle  Conjectures \ref{ConjectureHS} and \ref{ConjectureBetti}
in this case for all but finitely many values of the width.

\begin{thm}\label{Theorem4Generated}
Let $\G$ be a 4-generated numerical semigroup. 
Assume that $\w(\G) \geq 40$.
Then, we have
$b_i(R_\G) \leq i{\w(\G) +1 \choose i+1}$ for all $i \geq 1$.
\end{thm}

\section{Preliminaries}\label{SectionPreliminaries}

In this section, we fix the notation and state some preliminary results.
We refer to \cite{RosalesGarciaSanchez} for details on numerical semigroups,
and to \cite{HerzogHibi} for  free resolutions and monomial ideals.

Throughout this paper, the symbol $\Bbbk$ denotes a field.
Let $A$ be a finitely generated standard graded $\Bbbk$-algebra,
that is, an $\NN$-graded $\Bbbk$-algebra generated by finitely many elements of degree 1.
We denote  graded components by $[\cdot]_d$.
The Hilbert function of a graded $A$-module $M$  is  $\HF(M,d) = \dim_\Bbbk [M]_d$.
The Hilbert-Samuel function of $A$ is  $\HS(A,d) = \sum_{i=0}^d \HF(A,i)$.

If $A$ is a local ring with residue field $\Bbbk$, respectively,
 a finitely generated standard graded $\Bbbk$-algebra,
we denote by $\mm_A$ its unique maximal ideal, respectively, its unique maximal homogeneous ideal.
If $M$ is a finite $A$-module, respectively, a finite graded $A$-module, 
the (total) Betti numbers of $M$ over $A$ are defined as $b_i^A(M) = \dim_\Bbbk \mathrm{Tor}_i^A(M,\Bbbk)$.
The minimal number of generators of $M$ is denoted by $\mu(M)$, and we have $ \mu(M)=b_0^A(M)$.
Additionally, in the graded case we can define the graded Betti numbers of $M$ over $A$ as  
$b_{i,j}^A(M) = \dim_\Bbbk [\mathrm{Tor}_i^A(M,\Bbbk)]_j$.
If $M$ has finite length, we denote its length by $\ell(M)$.
In our context, the length always coincides with  the vector space dimension over $\Bbbk$.

If $A$ is a Noetherian local ring,
we denote by $\gr(A) = \bigoplus_{s\geq 0 } {\mm_A^s}/{\mm_A^{s+1}}$ its associated graded ring with respect to the maximal ideal, also known as the tangent cone.
It is a standard graded $\Bbbk$-algebra.

A numerical semigroup is a cofinite additive submonoid $\G \subseteq \NN$.
Every numerical semigroup has a unique minimal set of generators, 
and we write $\G = \langle g_0, g_1, \ldots, g_\nu \rangle$ to denote them.
We always write the generators in increasing order $g_0 < g_1 < \cdots < g_\nu$.
The smallest generator $g_0$ of $\G$ is  called  the  multiplicity of $\G$,
and we denote it by $m(\G)$ or simply by $m$.
It is well-known that $\nu+1 \leq m$.
Every nonzero additive submonoid of $\NN$ is isomorphic to a numerical semigroup, obtained by scaling all the elements by their greatest common divisor.

The semigroup ring of $\G$ is $R_\G = \Bbbk\llbracket t^\gamma \, : \, \gamma \in \G\rrbracket = 
\Bbbk\llbracket t^{m}, t^{g_1}, \ldots, t^{g_\nu} \rrbracket \subseteq \Bbbk\llbracket t \rrbracket$.
It is a one-dimensional Cohen-Macaulay local domain of multiplicity  $m$.
The ideal $(t^{m}) \subseteq R_\G$ is the unique monomial minimal reduction of the maximal ideal.
We denote the quotient by  $\overline{R}_\G = R_\G/(t^m)$,
which is an artinian local $\Bbbk$-algebra with  $\ell(\overline{R}_\G)=m$.
The ring $\overline{R}_\G$ has a monomial $\Bbbk$-basis $\{t^\gamma \, : \, \gamma \in \Ap(\G)\}$,
where $\Ap(\G) = \{\gamma \in \G \, : \, \gamma-m \notin \G\}$ is called the 
Apéry set of $\G$.

Let $P = \Bbbk \llbracket x_0, x_1, \ldots, x_\nu \rrbracket$  and consider the  presentation $\pi_\G : P \to R_\G$ defined by $\pi_\G(x_i)=t^{g_i}$.
We have $R_\G = P / I_\G$, where $I_\G = \ker(\pi_\G)$ is  the toric ideal of the semigroup $\G$
$$
I_\G = \left( x_0^{\alpha_0}\cdots x_\nu^{\alpha_\nu}-x_0^{\beta_0}\cdots x_\nu^{\beta_\nu} \, \mid \, 
\sum_{i=0}^\nu \alpha_i g_i = \sum_{i=0}^\nu \beta_i g_i\right).
$$
Let $\overline{P}=P/(x_0) \cong \Bbbk \llbracket x_1, \ldots, x_\nu \rrbracket,$
then we have the presentation $\overline{R}_\G = \overline{P}/\overline{I}_\G$,
where $\overline{I}_\G = \frac{I_\G+(x_0)}{(x_0)}$. 

Let $Q = \Bbbk[ x_1, \ldots, x_\nu]\cong \gr(\overline{P})$,
then we have a presentation 
$\gr(\overline{R}_\G) = Q/(\overline{I}_\G)^*$, 
where $(\overline{I}_\G)^*$ is the ideal of initial forms of $\overline{I}_\G$,
which is  generated by binomials and monomials
\begin{align*}
(\overline{I}_\G)^* = & \left(
x_1^{\alpha_1}\cdots x_\nu^{\alpha_\nu}-x_1^{\beta_1}\cdots x_\nu^{\beta_\nu}
\, \mid \, \sum_{i=1}^\nu \alpha_i g_i = \sum_{i=1}^\nu \beta_i g_i  \in \Ap(\G)
\,\text{ and }\,
\sum_{i=1}^\nu \alpha_i = \sum_{i=1}^\nu \beta_i
\right)\\
& +
\left(
x_1^{\alpha_1}\cdots x_\nu^{\alpha_\nu}
\, \mid \, 
\sum_{i=1}^\nu \alpha_i g_i = \sum_{i=1}^\nu \beta_i g_i \in \Ap(\G), \sum_{i=1}^\nu \alpha_i < \sum_{i=1}^\nu \beta_i
\text{ for some } \beta_1, \ldots, \beta_\nu \in \NN
\right)\\
&+
\left(
x_1^{\alpha_1}\cdots x_\nu^{\alpha_\nu}
\, \mid \, 
\sum_{i=1}^\nu \alpha_i g_i \notin \Ap(\G)
\right).
\end{align*}

The  presentations $R_\G = P /I_\G$, $\overline{R}_\G = \overline{P}/\overline{I}_\G$, 
and $\gr(\overline{R}_\G) = Q/(\overline{I}_\G)^*$ 
are called minimal regular presentations,
since they are  of the form $A= B/I$ where $B$ is a local or graded regular ring and $I \subseteq \mm_B^2$.
In these cases,
the Betti numbers of $A$ over $B$ are simply called the Betti numbers of $A$,  and denoted by $b_i(A) = b_i^B(A)$.
Moreover, we have $b_i^B(A) = b_{i-1}^B(I)$ for $i >0$.
Since $t^m$ is a non-zerodivisor on $R_\G$, we have $b_i(R_\G) = b_i(\overline{R}_\G)$ for every $i$.
Moreover, it is well-known that 
$b_i(\overline{R}_\G) \leq b_i(\gr(\overline{R}_\G))$ for every $i$.

We conclude this section by stating two theorems that provide 
upper bounds for the Betti numbers of an ideal in terms of its numerical invariants.

The first theorem identifies an  ideal with the largest graded Betti numbers among all the homogeneous ideals with a fixed Hilbert function.
Let  $S = \Bbbk[x_1, \ldots, x_n]$ be a polynomial ring.
A lexsegment ideal $L \subseteq S$ is  a monomial ideal such that, for each $d$, the  graded component $[L]_d$ is spanned  by the first $\dim_\Bbbk[L]_d$ monomials of  $[S]_d$,
 where monomials are ordered lexicographically.
For any homogeneous ideal $I \subseteq S$, there exists a unique lexsegment ideal,
denoted by $\Lex(I)\subseteq S$,
 such that $\HF(I) = \HF(\Lex(I))$,
cf. \cite[Theorem 6.3.1]{HerzogHibi}.

\begin{thm}[Bigatti-Hulett-Pardue]\label{ThmBHP}
Let  $S = \Bbbk[x_1, \ldots, x_n]$ and let $I \subseteq S $ be a homogeneous ideal.
We have $b_{i,j}(I) \leq b_{i,j}(\Lex(I))$ for all $i,j$.
In particular, $b_{i}(I) \leq b_{i}(\Lex(I))$ for all $i$.
\end{thm}

\begin{proof}
See \cite[Theorem 7.3.1]{HerzogHibi} for the case $\mathrm{char}(\Bbbk) = 0$,
\cite[Theorem 31]{Pardue} for the general case.
\end{proof}

The second theorem identifies an ideal with the largest Betti numbers among all ideals defining Cohen-Macaulay quotient rings with a fixed multiplicity.
In particular, this applies to ideals defining artinian  rings with a fixed length.
An ideal $I \subseteq S $ is called very compressed if  $\mm_S^{s+1} \subseteq I \subseteq \mm_S^s$ for some $s \in \NN$.

\begin{thm}[Valla]\label{ThmValla}
Let $A$ be a regular local ring (resp., a polynomial ring), 
and let $I\subseteq A$ be an ideal (resp., a homogeneous ideal) 
of codimension $n$ such that  $A/I$ is a Cohen-Macaulay ring of  multiplicity $m$.
Let $S = \Bbbk[x_1, \ldots, x_n]$ and let $\mathcal{C}=\mathcal{C}(m,n)\subseteq S$ be the unique 
very compressed lexsegment ideal such that $S/\mathcal{C}$ is an artinian $\Bbbk$-algebra with $\ell(S/\mathcal{C})=m$.
Then, 
$
b^A_i(A/I) \leq b^S_i(S/\mathcal{C})
$
for every $i$.
\end{thm}

\begin{proof}
The theorem is proved in \cite[Theorem 4]{Valla},
under the assumption that $I \subseteq \mm_A^2$.
However, another proof is given in \cite[Theorem 3.7]{CavigliaSammartano},
without that assumption.
\end{proof}

We point out that, for our purposes, 
it will be useful  to consider non-minimal regular presentations when estimating Betti numbers,
and, thus,
to drop the assumption that $I \subseteq \mm_A^2$.

\begin{cor}\label{CorValla}
Let $\G$ be a numerical semigroup.
Then, for all  $i\geq 1$ we have 
$$b_i(R_\G) \leq i{m(\G) \choose i+1}.$$
\end{cor}

\begin{proof}
Let $\pi_\G : P = \Bbbk \llbracket x_0, x_1, \ldots, x_\nu \rrbracket \twoheadrightarrow R_\G$  be the minimal regular presentation introduced above.
We have $\nu+1 \leq m$, thus, we can add variables and extend $\pi_\G$ to a  non-minimal presentation 
$\pi'_\G : P' = \Bbbk \llbracket x_0, x_1, \ldots, x_\nu, x_{\nu+1}, \ldots, x_{m-1} \rrbracket \twoheadrightarrow R_\G$
defined by $\pi'_\G(x_i) = 0$ for $i > \nu$,
and we have
  $I'_\G = \ker( \pi'_\G) = I_\G P'+ (x_{\nu+1}, \ldots, x_m) \subseteq P'$.
Since the generators of $I_\G P'$ and  $(x_{\nu+1}, \ldots, x_m)$ involve disjoint sets of variables,
the  minimal free resolution of $I'_\G$ is the tensor product of the resolutions of the two ideals. 
In particular, we see that $b_i(R_\G)= b_i^P(R_\G) \leq b_i^{P'}(R_\G)$ for all $i$.
Since $I'_\G \subseteq P'$ has codimension $m-1$ and  $P'/I'_\G \cong R_\G$ is Cohen-Macaulay of multiplicity $m$,
by Theorem \ref{ThmValla} we deduce that $b_i^{P'}(R_\G) \leq b_i^S(S/\mathcal{C})$ for every $i$, where
$$
\mathcal{C}=\mathcal{C}(m,m-1) = (x_1, \ldots, x_{m-1})^2 \subseteq S = \Bbbk[x_1,\ldots, x_{m-1}].
$$ 
The ideal $\mathcal{C}$ is resolved by an Eagon-Northcott complex, and we have $b_i^S(S/\mathcal{C}) = i{m \choose i+1}
$.
\end{proof}

The upper bounds of Corollary \ref{CorValla} are sharp,
since they are attained, for example, by the numerical semigroups
$\langle m, m+1, \ldots, 2m-1\rangle$ for all $m \geq 2$,
see for instance \cite[Proposition 2.7]{HerzogStamate}.

As an immediate consequence of Corollary \ref{CorValla},
the problem of bounding the Betti numbers of $R_\G$ in terms of the width $\w(\G)$ is 
solved when $\w(\G) \geq m(\G)-1$.

\begin{cor}\label{CorLargeWidth}
Let $\G$ be a numerical semigroup 
with  $\w(\G) \geq m(\G)-1$.
Then, for all $i\geq 1$
$$b_i(R_\G) \leq  i {\w(\G)+1 \choose i+1}.$$
\end{cor}

For the rest of the paper,
in view of Corollary \ref{CorLargeWidth}, 
we assume $\w(\G) \leq m(\G)-2$.

\section{Interval completion, monomial ideals, and Hilbert-Samuel function}\label{SectionMonomialIdeals}

To each numerical semigroup, we associate an artinian monomial ideal, as follows.

\begin{definition}
Let $\G$ be a numerical semigroup. 
With notation as in Section \ref{SectionPreliminaries},
we define  
$$
J_\G
= \mathrm{in}_{\mathrm{revlex}}\big((\ovI_\G)^*\big)
\subseteq Q= \Bbbk[x_1, \ldots, x_\nu],
$$ 
that is, the initial ideal of $(\ovI_\G)^*\subseteq Q $ with respect to the revlex monomial order.
\end{definition}

\noindent
We are interested in this monomial ideal because it gives upper bounds for the Betti numbers of the semigroup ring:
it is well-known that
passing to the artinian reduction, then to the associated graded ring, and finally to the initial ideal, yields the relations
\begin{equation}\label{EqBoundIGJG}
b_i(I_\G)  = b_i(\ovI_\G) \leq b_i((\ovI_\G)^*) \leq b_i(J_\G).
\end{equation}

\begin{example}
Let $\G = \langle 7,9,12,15\rangle$. 
Using  \cite{M2},
we compute the ideals associated to $\G$:
\begin{align*}
I_\G &=  \big(
x_0^3-x_1x_2,\, x_2^2-x_1x_3,\, x_1^3-x_2x_3,\, x_1^2x_2-x_3^2 
\big)
\subseteq P = \Bbbk\llbracket x_0, x_1, x_2, x_3\rrbracket,\\
\ovI_\G & = \frac{I_\G + (x_0)}{(x_0)} = 
\big(
x_1x_2,\, x_2^2-x_1x_3,\, x_1^3-x_2x_3,\, x_3^2 
 \big)
 \subseteq \overline{P} = \Bbbk \llbracket x_1,x_2,x_3\rrbracket,\\
(\ovI_\G)^* & = 
\big(
 x_1x_2,\,  x_2^2-x_1x_3,\, x_2x_3,\, x_3^2,\,  x_1^4 
 \big)
 \subseteq Q = \Bbbk[x_1,x_2,x_3],\\
J_\G & = 
\mathrm{in}_{\mathrm{revlex}}\big((\ovI_\G)^*\big) = 
\big(
x_1x_2,\, x_2^2,\,  x_2x_3,\, x_3^2,\,    x_1^4,\, x_1^2x_3
 \big)
 \subseteq Q.
\end{align*}
\end{example}

In this section, we prove  some constraints on the structure of the  ideal $J_\G$.

\begin{thm}\label{ThmJG}
Let $\G$ be a numerical semigroup with multiplicity $m$ and width $w$.
Assume that $w \leq m-2$.
Then, we have $J_\G \subseteq (x_1, \ldots, x_{\nu-1})^2 + x_\nu^q(x_1,\ldots,x_\nu)$,
where  $q = \lfloor \frac{m-1}{w}\rfloor$. 
Moreover, the Hilbert-Samuel function satisfies $\HS(Q/J_\G,d)\leq 1+dw$ for all $d \in \NN$.
\end{thm}
\noindent
The proof of Theorem \ref{ThmJG} relies on the arithmetic of numerical semigroups.
In the next sections, we leave the context of numerical semigroups, 
and we  study more generally ideals that satisfy these constraints, 
estimate their syzygies, and conclude the proof of the main theorems.

We fix the following notations and assumptions for this section.
Let  $\G$  be a numerical semigroup with $m(\G)=m$, $\w(\G)=w$, and $w \leq m-2$.
The interval completion of $\G$, introduced in \cite{HerzogStamate},
is the numerical semigroup 
\begin{equation}\label{EqGTilde}
\tilde{\G}=  \langle m, m+1, m+2, \ldots, m+w\rangle.
\end{equation}
Since $w \leq m-2$,
the above set of generators of $\tilde{\G}$ is minimal.
All minimal generators of $\G$ are also minimal generators of $\tilde{\G}$,
and
we have  $\w(\G) = \w(\tilde{\G})$ and $m(\G) = m(\tilde{\G})$.
We are going to consider the rings and ideals of Section \ref{SectionPreliminaries} for both semigroups $\G$ and $\tilde{\G}$,
thus, we add the subscript $\G$ or $\tilde{\G}$ to distinguish them.
We use   variables $x_0, x_1, \ldots, x_\nu$ for $\G$ and $y_0, y_1, \ldots, y_{w}$ for $\tilde{\G}$.
There is  a commutative diagram
\begin{center}
\begin{tikzcd}
P_\G = \Bbbk \llbracket x_0, \ldots, x_\nu \rrbracket \arrow[twoheadrightarrow]{r}{\pi_\G}  \arrow[hookrightarrow]{d}{\eta}  & \,\,R_\G \arrow[hookrightarrow]{d}{\iota}  \\
P_\tG = \Bbbk \llbracket y_0, \ldots, y_w \rrbracket  \arrow[twoheadrightarrow]{r}{\pi_\tG}
& \,\, R_\tG
\end{tikzcd}
\end{center}where $\iota : R_\G \hookrightarrow R_\tG$ is the inclusion induced by $\G \subseteq \tilde{\G}$,
and $\eta : P_\G \hookrightarrow P_\tG$ is the inclusion defined by setting $\eta(x_i) = y_j$ for $g_i = m+j$.
We have $\eta(x_0)= y_0, \eta(x_\nu) = y_w$,
and, since the order of the variables is preserved,
$\eta$ preserves inequalities of monomials with respect to the revlex orders in 
$Q_\G$ and $Q_\tG$.

The commutativity of the diagram implies that
$
I_\G = I_\tG \cap P_\G,
$
where, by a standard abuse of notation,
 the symbol $\cap$ denotes the pullback via $\eta$.
Applying the artinian reduction, we see that
$
I_\G +(x_0) = (I_\tG \cap P_\G) + (x_0) \subseteq \big(I_\tG + (y_0)\big) \cap P_\G, 
$
hence, that $\ovI_\G \subseteq \ovI_\tG \cap \ovP_\G$.
Taking ideals of initial forms preserves inclusions, 
thus, we have
$(\ovI_\G)^* \subseteq (\ovI_\tG)^* \cap \ovP_\G$.
Similarly, 
the inclusion is  preserved when
taking initial ideals with respect to revlex
in $Q_\G$ and $Q_\tG$, 
thus, we obtain the inclusion of monomial ideals
\begin{equation}\label{EqInclusionOfMonomialIdeals}
J_\G \subseteq J_\tG \cap Q_\G,
\end{equation}
where, again, 
 $\cap$ denotes the pullback via the inclusion $\eta : Q_\G \hookrightarrow Q_\tG$ defined as above.

We determine the  ideal $ J_\tG$ explicitly.

\begin{prop}\label{PropJGTilde}
Let $q, r \in \NN$ be the integers such that $m = qw+r$ and $1 \leq r \leq w$.
Then,
$$
J_\tG = (y_1, \ldots, y_{w-1})^2 + y_w^q(y_r, \ldots, y_w). 
$$
\end{prop}
\begin{proof}
Inspecting the generators \eqref{EqGTilde} of $\tilde{\G}$, 
we can determine some binomials in the toric ideal $I_\tG$.
We have
$
\mathfrak{f}_{h,i,j,k} := y_hy_i - y_j y_k\in I_\tG
$
for all indices 
$0 \leq h,i,j,k \leq w$
such that 
$h+i = j+k$,
 and 
$
\mathfrak{g}_i :=
y_w^qy_{i+r}-y_0^{q+1}y_i
\in I_\tG
$ 
for
$
0 \leq i \leq w-r.
$
Going modulo $y_0$, it follows that 
$\ovI_\tG$ contains 
the monomials $\overline{\mathfrak{f}}_{h,i,0,h+i} =  y_h y_i$ for 
all $1 \leq h,i\leq w-1$ such that $h+i\leq w$,
the binomials $\overline{\mathfrak{f}}_{h,i,j,k} =  y_hy_i - y_j y_k$  for all indices $1 \leq h,i,j,k \leq w$
such that $h+i = j+k$,
 and the monomials
$
\overline{\mathfrak{g}}_i =
y_w^qy_{i+r}
$ 
for
$
0 \leq i \leq w-r.
$
All monomials and binomials listed in the previous sentence are homogeneous polynomials,
so they coincide with their initial forms,
and thus they also belong to $(\ovI_\tG)^*$.
Finally,
for any  indices $1 \leq j < h \leq i < k\leq w$
such that $h+i = j+k$,
the initial term of  
$\overline{\mathfrak{f}}_{h,i,j,k} =  y_hy_i - y_j y_k$ 
with respect to the revlex order is $y_hy_i$.
Thus, we obtain the inclusion of ideals of $Q_\tG = \Bbbk[x_1, \ldots, x_w]$
$$
H := (y_1, \ldots, y_{w-1})^2 + y_w^q(y_r, \ldots, y_w) \subseteq \mathrm{in}_{\mathrm{revlex}}\big((\ovI_\tG)^*) =  J_\tG.
$$
The Hilbert function of 
the artinian graded $\Bbbk-$algebra $ Q_\tG/H$ is 
\begin{equation}\label{EqHFtG}
\HF( Q_\tG/H,e) = \begin{cases}
1 & \text{ if } e=0,\\
w & \text{ if } 1\leq e\leq q,\\
r-1 & \text{ if } e=q+1,\\
0 & \text{ if } e\geq q+2,\\
\end{cases}
\end{equation}
in particular,  $\ell(Q_\tG/H)=qw+r=m$.
On the other hand, passing to the associated graded ring and to the initial ideal preserves  length, 
so $\ell(Q_\tG/J_\tG)=\ell(\ovR_\tG)=m$.
Thus, we must have $H = J_\tG$.
 \end{proof}
 
Next, we establish a  comparison between the Hilbert-Samuel functions of 
$\gr(\ovR_\G)$ and  $\gr(\ovR_\tG)$.

\begin{prop}\label{PropHilbertSamuelInequality}
For every $d \in \NN$, we have
$\HS(\gr(\ovR_\G),d) \leq \HS(\gr(\ovR_\tG),d)$.
\end{prop}
\begin{proof}
We recall that $\ovR_\G$, and hence $\gr(\ovR_\G)$, have the $\Bbbk$-basis 
$\{ t^\gamma \, \mid \, \gamma \in \Ap(\G)\}$.
The Apéry set of $\G$ is equal to $
\Ap(\G)  = \big\{\omega_j(\Gamma) \, \mid \, j = 0, \ldots, m-1\big\},
$
where 
$$
\omega_j(\Gamma)=\min \{\omega \in \G \, \mid \, \omega \equiv j \pmod{m}\}.
$$
The degree of a monomial $t^\gamma \in \gr(\ovR_\G)$ is equal to the largest $s \in \NN$ such that $t^\gamma \in \mm_{R_\G}^s$,
equivalently, it is equal to  the so-called order of the element $\gamma\in \G$, defined as 
$$
\ord_\G(\gamma) = \max\left\{ \sum_{i=1}^\nu \alpha_i \, \mid\, \gamma = \sum_{i=1}^\nu \alpha_i g_i\right\}.
$$
It follows that $\HS(\gr(\ovR_\G),d) = \#\, \{ j  \, \mid \, \ord_\G(\omega_j(\G))\leq d\}$.
In order to prove the proposition, it suffices to show that $\ord_\tG(\omega_j(\tG)) \leq \ord_\G(\omega_j(\G))$ for all $j = 0, \ldots, m-1$.

Since $\w(\G) \leq m-2$,  the generators of $\G$ are of the form $g_i = m+j_i$ for some $0\leq j_i \leq w$.
We denote $\Omega = \{j_i \,\mid \, i = 1, \ldots, \nu \}\subseteq\{0, \ldots, w\}$.
Fix an $h \in \{0,\ldots,m-1\}$ and consider the elements 
$\omega_h(\G) \in \Ap(\G)$ and $\omega_h(\tG) \in \Ap(\tG)$.
Let $\alpha = \ord(\omega_h(\G))$ and $\beta = \ord(\omega_h(\tG))$.
By definition of order and $\Omega$, there exist $\alpha_j, \beta_j \in \NN$ such that
\begin{align*}
\omega_h(\G) = \sum_{j = 1 }^w \alpha_j (m+j), & \quad \sum_{j=1}^w \alpha_j =\alpha,\quad  \alpha_j = 0 \text{ for } j \notin \Omega,\\
\omega_h(\tG) = \sum_{j=1}^w \beta_j (m+j), & \quad \sum_{j =1}^w \beta_j =\beta.
\end{align*}

Assume by contradiction that $\ord(\omega_h(\tG)) > \ord(\omega_h(\G))$,
that is, $\beta>\alpha$.
Since $\G \subseteq \tG$,
it follows from the definition of $\omega_h(\cdot)$ that $\omega_h(\tG)  \leq \omega_h(\G)$.
We obtain the inequalities
\begin{align*}
\omega_h(\tG) \leq \sum_{j=1}^w \alpha_j (m+j)
\leq \sum_{j=1}^w \alpha_j (m+w)= \alpha(m+w) 
\leq (\beta-1)(m+w) \leq \beta m +(\beta-1) w
\end{align*}
On the other hand, we have $\omega_h(\tG) \geq \sum_{j=1}^w \beta_j m = \beta m$, therefore,
we can write $\omega_h(\tG) = \beta m + \delta$ for some $\delta \in \NN$ such that $0 \leq \delta \leq (\beta-1)w$.
Write $\delta = (\beta-1)\zeta +\rho $ with $\zeta,\rho\in \NN$ and $\rho < \beta-1$.
We must have either $\zeta < w$ or $\zeta = w $ and $\rho =0$.
We have $\delta = \rho ( \zeta+1) + (\beta-1-\rho)\zeta$ and, thus,
\begin{equation}\label{EqNotInAperySet}
 \omega_h(\tG) = \beta m + \rho ( \zeta+1) + (\beta-1-\rho)\zeta = m + \rho(m+\zeta+1) + (\beta-1-\rho)(m+\zeta).
\end{equation}
Note that $m + \zeta\in \tG$, and if $\rho \ne 0$ then we also have $m+\zeta+1 \in \tG$.
In any case, equation \eqref{EqNotInAperySet} shows that $ \omega_h(\tG)-m \in \tG$, contradicting 
the fact that $\omega_h(\tG)\in \Ap(\tG)$.
\end{proof}

We can now complete the proof of Theorem \ref{ThmJG}.

\begin{proof}[Proof of Theorem \ref{ThmJG}]
Combining equation \eqref{EqInclusionOfMonomialIdeals} and Proposition \ref{PropJGTilde}, we obtain
\begin{align*}
J_\G &\subseteq  \big((y_1, \ldots, y_{w-1})^2 + y_w^q(y_r, \ldots, y_w)\big) \cap  Q_\G\\
&\subseteq  \big((y_1, \ldots, y_{w-1})^2 + y_w^q(y_1, \ldots, y_w)\big) \cap  Q_\G\\
&= (x_1, \ldots, x_{\nu-1})^2 + x_\nu^q(x_1, \ldots, x_\nu),
\end{align*}
verifying the first claim.
By Proposition \ref{PropHilbertSamuelInequality}, we have $\HS(\gr(\ovR_\G),d) \leq \HS(\gr(\ovR_\tG),d)$ for all $d\in \NN$.
By equation \eqref{EqHFtG}, it follows that
$$
\HS\big(\gr(\ovR_\tG),d\big)
= \sum_{e=0}^d 
\HF\big(\gr(\ovR_\tG),e\big)
= 
\begin{cases}
1+dw & \text{ if }0 \leq  d \leq q,\\
m & \text{ if } d \geq q+1,
\end{cases}
$$
verifying the second claim.
\end{proof}

\section{Estimating the  syzygies}\label{SectionSyzygies}

Now that we have obtained Theorem \ref{ThmJG}, 
we can attack Question \ref{QueBoundWidth} by enlarging the class of ideals  and considering a more general problem:
   bound the Betti numbers of arbitrary homogeneous ideals 
of finite (but unspecified) colength, satisfying a constraint on the Hilbert-Samuel function.

\begin{problem}\label{ProblemArbitraryHomogeneous}
Let $S = \Bbbk[x_1, \ldots, x_n]$
and $I\subseteq S$ be a homogeneous ideal such that $\ell(S/I) < \infty$.
Let $w \in \NN$ be a positive integer.
Find upper bounds for the Betti numbers $b_i(S/I)$
under the assumption that
\begin{equation}\label{EqConstraintHS}
\HS(S/I,d) \leq 1+dw
\qquad \text{ for all }
d \in \NN.
\end{equation}
\end{problem}

Problem \ref{ProblemArbitraryHomogeneous} can be seen as a variant of Theorems \ref{ThmBHP} and \ref{ThmValla}, and it is interesting in its own right.
Both the Hilbert-Samuel function and the Betti numbers give rise to natural stratifications in  Hilbert schemes parametrizing artinian algebras, and it is often fruitful to study the interplay   between the two invariants.

In studying Problem \ref{ProblemArbitraryHomogeneous},
by Theorem \ref{ThmBHP}, 
we can restrict to lexsegment ideals.
The syzygies of a lexsegment ideal $L$
are governed by the colength of their ``hyperplane section'' $\frac{L+(x_n)}{(x_n)}$.
To this purpose,
in this section, we fix the following notation.
Let $S = \Bbbk[x_1, \ldots, x_n]$.
We denote by $\widehat{\cdot}$ images modulo $x_n$.
Thus, we have $\wS=S/(x_n) \cong   \Bbbk[x_1, \ldots, x_{n-1}]$,
and for ideals $I \subseteq S$ we denote by $\wI$ their image in $\wS$.

\begin{lemma}\label{LemmaSyzygiesHyperplane}
Let $L \subseteq S$ be a lexsegment  ideal such that $\ell(S/L) < \infty$.
Then, for all $i \geq 0$ we have
$$
b_i^S\big(S/L\big) \leq  \ell\big(\wS/\wL\big){n \choose i}.
$$
\end{lemma}
\begin{proof}
We prove the lemma in two steps.
First, we claim that, for all  $i \geq 0$, we have
\begin{equation}\label{EqBettiHyperplane1}
b_i^S\big(L\big) = b_i^{\wS}\big(\wL\big) + \ell\big(\wS/\wL\big){n-1 \choose i}.
\end{equation}
The argument is similar to that of
\cite[Proof of Theorem 3.7]{CavigliaSammartano}.
Since $x_n$ is a non-zerodivisor on $S$ and $L$, we have
 $b_i^S\big(L\big) = b_i^{\wS}\big(L/x_nL\big)$.
Decompose $L = \oplus_{h=0}^\infty L_h x_n^h$, where $L_h \subseteq \wS$ are monomial ideals.
 Note that $L_0 = \wL$ and $L_h = \wS$ for $h \gg 0$.
It follows that we have a decomposition of $\wS$-modules 
\begin{equation}\label{EqDecompositionL}
 \frac{L}{x_nL}=  L_0 \oplus \bigoplus_{h=0}^\infty \frac{L_{h+1}}{L_h} \cong \wL \oplus \Bbbk^{\ell\big(\wS/\wL\big)}.
\end{equation}
To see the last isomorphism, 
note that the lexsegment property of $L$ implies  $(x_1, \ldots,x_{n-1}) L_{h+1} \subseteq L_h$ for every $h$.
Thus, the $\wS$-module 
$\bigoplus_{h=0}^\infty \frac{L_{h+1}}{L_h}$
is annihilated by $\mm_{\wS}$, 
and it is a direct sum of copies of $\Bbbk$.
The number of copies is $\sum_{h=0}^\infty \ell({L_{h+1}}/{L_h}) = \ell(L_h/L_0)$ for $h \gg 0$, that is, $\ell(\wS/\wL)$.
Now, formula \eqref{EqBettiHyperplane1} follows  from \eqref{EqDecompositionL} since $b_i^{\wS}(\Bbbk) = {n-1 \choose i}$ by the Koszul resolution.

Second, we claim that, for all $i \geq 0 $, we have
\begin{equation}\label{EqBettiHyperplane2}
b_i^{\wS}\big(\wS/\wL\big) \leq \ell\big(\wS/\wL\big){n-1 \choose i},
\end{equation}
We prove the claim by induction on $\ell\big(\wS/\wL\big)$.
If $\ell\big(\wS/\wL\big)=1$, then $\wL = \mm_{\wS}$, and \eqref{EqBettiHyperplane2} follows from the Koszul resolution.
Suppose $\ell\big(\wS/\wL\big)>1$, and let $I \subseteq \wS$ be a lexsegment ideal such that $\wL \subseteq I$ and $\ell(I/\wL)=1$.
It follows that $I /\wL \cong \Bbbk$ as $\wS$-module, and  $\ell(\wS/I) = \ell(\wS/\wL)-1$.
Applying the long exact sequence of $\mathrm{Tor}^{\wS}(\cdot, \Bbbk)$ (or the horseshoe lemma) to the short exact sequence 
$
0 \to {I}/{\wL} \to {\wS}/{\wL} \to {\wS}/{I} \to 0,
$
and the induction hypothesis,
we obtain the claim 
$$
b^{\wS}_i(\wS/\wL) \leq b^{\wS}_i(I/\wL) +b^{\wS}_i(\wS/I) = {n-1 \choose i} +b^{\wS}_i(\wS/I) \leq {n-1 \choose i} +\ell(\wS/I) {n-1 \choose i} = \ell(\wS/\wL) {n-1 \choose i}.
$$

Finally, combining \eqref{EqBettiHyperplane1} and \eqref{EqBettiHyperplane2} we conclude the proof of the lemma: 
if $i \geq 1$, then
\begin{align*}
b_i^S\big(S/L\big) 
&= b_{i-1}^S\big(L\big)  =b_{i-1}^{\wS}\big(\wL\big) + \ell\big(\wS/\wL\big){n-1 \choose i-1} 
= b_{i}^{\wS}\big(\wS/\wL\big) + \ell\big(\wS/\wL\big){n-1 \choose i-1} \\
&\leq  \ell\big(\wS/\wL\big){n-1 \choose i} + \ell\big(\wS/\wL\big){n-1 \choose i-1} 
= \ell\big(\wS/\wL\big){n \choose i},
\end{align*}
whereas the conclusion of the lemma is obvious for $i = 0$.
\end{proof}

\begin{prop}\label{PropEstimateHyperplaneSection}
Let $S =  \Bbbk[x_1, \ldots, x_w]$, with $w \geq 3$.
Let $L \subseteq S $ be a lexsegment ideal such that
 $\HS(S/L,d) \leq 1+dw$ for all $d \in \NN$.
Then, we have
$
\ell(\wS/\wL) \leq (3e)^{\sqrt{2w}},
$
where $e$ is the base of natural logarithms.
\end{prop}
\begin{proof}
We use the assumption on the Hilbert-Samuel function to estimate the least pure power of $x_i$ in $L$  for each  $i = 1, \ldots, w-1$.
We obtain the following estimates:
\begin{enumerate}[label=(\alph*)]
\item 
$x_i ^2 \in L$ for  $1 \le i \le w-2\sqrt{w}+2.$ 
\\
Assume by contradiction $x_{i}^{2}\notin L$.
By the lexsegment property  $L$ contains no monomial in $x_i, \ldots, x_w$ of degree at most 2.
Using the Hilbert-Samuel inequality for $d=2$, we obtain a contradiction
\begin{align*}
1+2w &\geq \HS(S/L,2) \geq \HS(\Bbbk[x_i, \ldots, x_w],2) = \binom{w-i+3}{2}
\\
&=\frac{(w-i+3)(w-i+2)}{2} \ge \frac{(2\sqrt{w}+1)2\sqrt{w}}{2}=2w+\sqrt{w}.
\end{align*}

\item $x_i^D \in L$ for $1 \leq i \leq w-2$, where $D= \lfloor \sqrt{6w-2}\rfloor-2$.
\\
By the lexsegment property, it suffices to show that $x_{w-2}^D \in L$.
Assume by contradiction that $x_{w-2}^D \notin L$, then $L$ contains no monomial in $x_{w-2},x_{w-1},x_w$ of degree at most $D$.
Using the Hilbert-Samuel inequality for $d=D$, we obtain the inequalities
\begin{align*}
1+Dw &\geq \HS(S/L,D) \geq \HS(\Bbbk[x_{w-2},x_{w-1},x_w],D) = \binom{D+3}{3} = \frac{(D+3)(D+2)(D+1)}{6}\\
& \Rightarrow 6+6Dw \geq D^3+6D^2+11D+6 \Rightarrow D^2+6D+(11-6w) \leq 0.
\end{align*}
Solving the quadratic inequality, we find $D \leq \sqrt{6w-2}-3$, contradiction.

\item 
$x_{w-1}^{2w-2}	\in L$. 
\\
Assume by contradiction that $x_{w-1}^{2w-2}\notin L$,
then $L$ contains no monomial in $x_{w-1},x_w$ of degree at most $2w-2$.
Using the Hilbert-Samuel inequality for $d=2w-2$ we obtain a contradiction
\begin{align*}
1+(2w-2)w &\geq \HS(S/L,2w-2) \geq \HS(\Bbbk[x_{w-1},x_w],2w-2) ={2w \choose 2}=2w^2-w.
\end{align*}
\end{enumerate}
Denote $C = \lceil 2 \sqrt{w}\rceil -2$.
Going modulo $x_w$, the estimates (a), (b), (c) and the lexsegment property yield the inclusion of ideals
$
I = (x_1, \ldots, x_{w-C})^2 + (x_1, \ldots, x_{w-2})^D + (x_1, \ldots, x_{w-1})^{2w-2} \subseteq \wL,
$
hence,
the inequality $\ell(\wS/\wL) \leq \ell(\wS/I) = \HS(\wS/I,2w-3)$.
We compute the latter by inclusion-exclusion
\begin{align}
 \HS(\wS/I,2w-3)=& \,\,
\HS(\Bbbk[x_1,\ldots,x_{w-1}],1) 	\nonumber\\
&+\HS(\Bbbk[x_{w-C+1},\ldots,x_{w-1}],D-1)-\HS(\Bbbk[x_{w-C+1},\ldots,x_{w-1}],1) \nonumber
\\
&+\HS(\Bbbk[x_{w-1}],2w-3)-\HS(\Bbbk[x_{w-1}],D-1)
\nonumber
\\
\label{EqBinomialCD}
=&\,\,  w + {C+D-2 \choose C-1} - (C-1) + (2w-4)-D.
\end{align}
Next, we use the well-known inequality\footnote{It follows by considering
${A \choose B}= \frac{A(A-1)\cdots(A-B+1)}{B!}\leq \frac{A^B}{B!}
=\left(\frac{A}{B}\right)^B\frac{B^B}{B!}
\leq \left(\frac{A}{B}\right)^B\sum_{k=0}^\infty \frac{B^k}{k!}
= \left(\frac{A}{B}\right)^Be^B.
$
}  ${A \choose B} \leq  \left(\frac{eA}{B}\right)^B$, 
where $A\geq B$ are positive integers.
Thus, we have inequalities
\begin{align}
\ell(\wS/\wL)  &\leq
3w-3-C-D+\left(e\frac{C+D-2}{C-1}\right)^{C-1} \nonumber\\
&= 3w - \lceil \sqrt{2w}\rceil - \lfloor \sqrt{6w-2}\rfloor +1 + 
\left(e\frac{\lceil \sqrt{2w}\rceil
+\lfloor \sqrt{6w-2}\rfloor-6}{\lceil \sqrt{2w}\rceil-3}\right)^{\lceil \sqrt{2w}\rceil-3}\nonumber\\
& \label{EqFunctionW} \leq 3w +
\left(e\frac{ \sqrt{2w}
+\sqrt{6w-2}-5}{ \sqrt{2w}-3}\right)^{ \sqrt{2w}-2}.
\end{align} 
It is straightforward to verify  that, if $w \geq 112$, then we have inequalities  
$$3(\sqrt{2w}-3) \geq \sqrt{2w}
+\sqrt{6w-2}-5
\qquad \text{and} \qquad
((3e)^2-1)(3e)^{\sqrt{2w}-2} \geq 3w,
$$
so we can bound \eqref{EqFunctionW} by 
$$
3w + (3e)^{ \sqrt{2w}-2}\leq (3e)^{ \sqrt{2w}}.
$$
On the other hand, we verify directly that the quantity \eqref{EqBinomialCD} is bounded above by 
$(3e)^{ \sqrt{2w}}$ for $3 \leq w \leq 111$,
thus concluding the proof that 
$\ell(\wS/\wL) \leq (3e)^{ \sqrt{2w}}$ for all $w \geq 3$.
\end{proof}

We have obtained the main result of this section,
which gives a partial solution to Problem \ref{ProblemArbitraryHomogeneous}.

\begin{thm}\label{TheoremBoundArbitraryHomogeneous}
Let $S = \Bbbk[x_1, \ldots, x_n]$
and $I\subseteq \mm_S^2$ be a homogeneous ideal such that $\ell(S/I) < \infty$.
Let $w \in \NN$ be a positive integer, and assume that 
$\HS(S/I,d) \leq 1+dw$
for all 
$d \in \NN$.
Then, we have $b_i(S/I) \leq {w \choose i} (3e)^{\sqrt{2w}}$ for all $i = 0,\ldots, n$.
\end{thm}
\begin{proof}
Since $I\subseteq \mm_S^2$ and $\HS(S/I,1) \leq 1+w$, we see that $n \leq w$.
We add new variables and consider the ideal $I'=IS' + (x_{n+1}, \ldots, x_w) \subseteq S' = \Bbbk[x_1, \ldots, x_w]$.
 As in the proof of Corollary \ref{CorValla},
we have $S'/I' = S/I$ and $b_i(S/I) = b_i^S(S/I) \leq  b_i^{S'}(S'/I')$.
The conclusion now follows from Theorem \ref{ThmBHP}, Lemma \ref{LemmaSyzygiesHyperplane} and Proposition \ref{PropEstimateHyperplaneSection}.
\end{proof}

We can now  conclude the proof of Theorem \ref{TheoremGeneralBounds}.

\begin{proof}[Proof of Theorem \ref{TheoremGeneralBounds}]
Let $\G$ be a numerical semigroup with $\w(\G) =w$.
By Herzog's theorem \cite{Herzog}, we may assume $\nu \geq 3$, hence, $w\geq 3$.
By Corollary \ref{CorLargeWidth}, we may assume $w \leq m(\G)-2$.
By \eqref{EqBoundIGJG}, we have $b_i(R_\G) \leq b_i(Q/J_\G)$.
By Theorems \ref{ThmJG} and \ref{TheoremBoundArbitraryHomogeneous},
we have $b_i(Q/J_\G)\leq {w \choose i } (3e)^{\sqrt{2w}}$.
\end{proof}

\section{Proof of the Herzog-Stamate conjecture for 4-generated numerical semigroups}\label{SectionFourGenerated}

In this section, we  concentrate on 4-generated semigroups, that is, we assume $\nu = 3$.
In this case, by refining the analysis of Section \ref{SectionSyzygies},
we are able to  prove Conjectures \ref{ConjectureHS} and \ref{ConjectureBetti} for large enough width.

\begin{thm}\label{TheoremBoundsTwoVariables}
Let $S =  \Bbbk[x,y,z]$ and $w \geq 40$.
Let $L \subseteq \mm_S^2 $ be a lexsegment ideal such that
 $\HS(S/L,d) \leq 1+dw$ for all $d \in \NN$.
Then, we have
$
\mu(L) = b_0(L) \leq {w+1 \choose 2}
$
and $b_1(L) \leq 2 {w+1 \choose 3}$.
\end{thm}

\begin{proof}
Let $x^\alpha, y^\beta$ denote the smallest pure powers of $x, y $ in $L$.
By assumption, we have $2\leq \alpha\leq \beta$.
Observe that $\wL \subseteq \wS = \Bbbk[x,y]$ is also a lexsegment ideal, thus,
it is of the form
\begin{equation}\label{EqLex2Vars}
\wL = \big(x^\alpha, x^{\alpha-1}y^{\beta_1+1}, x^{\alpha-2}y^{\beta_2+2}, \ldots, xy^{\beta_{\alpha-1}+\alpha-1}, y^\beta\big) = \sum_{i=0}^{\alpha-1} x^{\alpha-i}\big(y^{\beta_i+i}\big)+\big(y^\beta\big)
\end{equation}
for some  integers $0 \leq \beta_1 \leq \beta_2 \leq \cdots \leq \beta_{\alpha-1} \leq \beta-\alpha$.
It follows that  $\mu(\wL) = \alpha +1$, and, since $\wL$ has projective dimension one,  $b_1(\wL) = \alpha$.
Moreover, since $H=(x^\alpha)+(x,y)^\beta \subseteq \wL$, we have the estimate
\begin{equation}\label{EqEstimateLength2Vars}
\ell(\wS/\wL) \leq \ell(\wS/H) = \HS(\wS/H,\beta-1 ) =
\HS(\wS,\alpha-1 ) + \sum_{d=\alpha}^{\beta-1}\HF(\wS/H,d) = {\alpha+1 \choose 2} + \alpha(\beta-\alpha).
\end{equation}

Using equation \eqref{EqBettiHyperplane1}, we obtain the following bounds for the Betti numbers of $L$
\begin{equation}\label{EqBoundsB}
\begin{split}
b_0(L) &= b_0(\wL) + \ell(\wS/\wL) \leq \alpha+1 + {\alpha+1 \choose 2} + \alpha(\beta-\alpha) = \alpha\beta  - \frac{\alpha(\alpha-3)}{2}+1,
\\
 b_1(L) &= b_1(\wL) + 2\ell(\wS/\wL) \leq \alpha + 2{\alpha+1 \choose 2} + 2\alpha(\beta-\alpha) = 2\alpha\beta -\alpha^2+2\alpha.
\end{split}
\end{equation}

We now use the assumption on the Hilbert-Samuel function twice, for $d = \alpha-1$  and $d = \beta-1$.
Since $L$ contains no monomials of degree at most $\alpha-1$, we have
\begin{align}
& {\alpha+2 \choose 3} = \HS(S,\alpha-1) = \HS(S/L,\alpha-1) \leq 1 + (\alpha-1)w  \label{EqAlphaMinusOne}
\\
\Rightarrow & (\alpha+2)(\alpha+1)\alpha \leq 6 +6(\alpha-1)w \nonumber
\\
\Rightarrow &
(\alpha+1)^2 + \alpha+1 = (\alpha+2)(\alpha+1) \leq 6+6w \nonumber
\\
\Rightarrow &
(\alpha+1)^2\leq 6w+6-\alpha-1\leq 6w+4 \nonumber
\\
\Rightarrow 
&\alpha \leq \sqrt{6w+4}-1. \label{EqBoundAlpha}
\end{align}
Similarly, $L$ contains no monomials in $y,z$ of degree at most $\beta-1$,
thus, by inclusion-exclusion we get
\begin{equation}\label{EqBetaMinusOne}
\begin{split}
1 + (\beta-1)w  &\geq \HS(S/L,\beta-1)  
\geq
\HS(S,\alpha-1) +\HS(\Bbbk[x,y],\beta-1)-\HS(\Bbbk[x,y],\alpha-1)\\
&={\alpha+2 \choose 3} +{\beta+1 \choose 2}-{\alpha+1 \choose 2}
= {\alpha+1 \choose 3} +{\beta+1 \choose 2},
\end{split}
\end{equation}
from which we obtain the inequalities
\begin{equation}\label{EqBoundBeta}
(\beta+1)\beta \leq 2 + 2(\beta-1)w \Rightarrow \beta+1\leq 2+2w\Rightarrow
\beta \leq 2w+1.
\end{equation}
Using  \eqref{EqBoundsB}, \eqref{EqBoundAlpha}, \eqref{EqBoundBeta} we obtain
\begin{equation}\label{EqBoundsB2}
\begin{split}
b_0(L) &\leq
 \alpha\beta -\frac{\alpha(\alpha-3)}{2}+1
\leq \alpha \beta +2 \leq (\sqrt{6w+4}-1) (2w+1)+2,  \\
b_1(L) &\leq  2\alpha\beta -\alpha^2+2\alpha\leq 2\alpha\beta +1
\leq 2(\sqrt{6w+4}-1) (2w+1). 
\end{split}
\end{equation}

It is straightforward to verify that, if $w \geq 100$, then we have 
the inequalities
$$
(\sqrt{6w+4}-1) (2w+1)+2 \leq 5 w^{\frac{3}{2}}\leq {w+1 \choose 2}
\quad
\text{and}
\quad
2(\sqrt{6w+4}-1) (2w+1) \leq 10 w^{\frac{3}{2}}\leq 2{w+1 \choose 3},
$$
yielding the desired bounds for $b_0(L)$ and $b_1(L)$, under the assumption that $w \geq 100$.
For each value  $40 \leq w \leq 99$ there are finitely many  $\alpha, \beta$ satisfying \eqref{EqAlphaMinusOne} and \eqref{EqBetaMinusOne}.
We verify directly that, for each such triple $(w,\alpha,\beta)$, 
the quantities on the right hand side of the
bounds \eqref{EqBoundsB} do not exceed 
${w+1 \choose 2}$ and $2{w+1\choose 3}$, respectively.
This concludes the proof.
\end{proof}

\begin{proof}[Proof of Theorem \ref{Theorem4Generated}]
It follows from Theorem \ref{TheoremBoundsTwoVariables}, as Theorem \ref{TheoremGeneralBounds} follows from Theorem \ref{TheoremBoundArbitraryHomogeneous}.
\end{proof}

\begin{remark}
We believe that the conclusions of Theorem \ref{TheoremBoundsTwoVariables} hold also if $4 \leq w \leq 39$,
and this would lead to a complete proof of Conjecture \ref{ConjectureBetti} for 4-generated semigroups.
However, 
our method does not work for such values of $w$.
For example, we note  that the estimate
\eqref{EqEstimateLength2Vars} becomes too  coarse if the $\beta_i$'s are small,
whereas \eqref{EqBetaMinusOne}  becomes too  coarse if  the $\beta_i$'s are large.
%One could try to refine the argument by accounting for 
A refinement of the argument would probably need to  account for
all possible sequences $\beta_1 \leq \cdots \leq \beta_{\alpha-1}$
and apply the  Hilbert-Samuel inequality  for all $\alpha \leq d \leq \beta-1$,
but this refined optimization problem appears to be intractable. 
Alternatively,
we  note that our proof identifies a finite number (about 200) of possible exceptions $(\alpha,\beta)$,
and, hence,
 a finite number of possible $\wL$ contradicting the conclusions. 
 Thus, 
 one could try to complete the proof by analyzing  each $\wL$ individually, with the help of a computer;
however, the issue with this approach is that the finite number of ideals $\wL$ is still too large to devise an effective computation.
\end{remark}

\section*{Acknowledgments}
The authors would like to thank Francesco Strazzanti for some helpful conversations.
The software Macaulay2 \cite{M2} was used for some computations related to this paper.

\section*{Funding}
G. C.  was partially supported by a grant from the Simons Foundation (41000748, G.C.).
A. M. was supported by 
the grant “Proprietà locali e globali di
anelli e di varietà algebriche” PIACERI 2020–22, Università degli Studi di Catania.
A. S.  was partially supported by the grant
 PRIN 2020355B8Y “Squarefree Gr\"oner degenerations, special varieties and related topics”
 and by the INdAM – GNSAGA Project  CUP E55F22000270001.

\end{document}